\documentclass[12pt,a4paper]{amsart}

\usepackage[utf8]{inputenc}
\usepackage[T1]{fontenc}
\usepackage{lmodern}
\renewcommand{\bfdefault}{bx}
\usepackage{amsmath,amsfonts,amssymb}
\usepackage{mathrsfs}
\usepackage{microtype}
\usepackage{verbatim}
\usepackage{graphicx}
\usepackage{xcolor}
\usepackage{enumerate}
\usepackage[toc]{appendix}
\usepackage{mdframed}
\usepackage[breakable]{tcolorbox}
\tcbuselibrary{many}
\usepackage{tikz}
\usepackage{fancyhdr}
\usepackage{cite}

\IfFileExists{dutchcal.sty}{\usepackage{dutchcal}}{}
\IfFileExists{dsfont.sty}{\usepackage{dsfont}}{%
}

\usepackage[hidelinks,pdfencoding=auto]{hyperref}
\definecolor{citation}{rgb}{0.11,0.67,0.84}
\definecolor{formula}{rgb}{0.1,0.2,0.6}
\definecolor{url}{rgb}{0.11,0.67,0.84}
\hypersetup{
  pdftitle={The Singular Set of Minima of Multiple Integrals},
  pdfauthor={Cristiana De Filippis, Jan Kristensen, Giuseppe Mingione, Aidan Strong},
  pdfsubject={Fractional differentiability for quasiconvex integrals with x and u dependence}
}

\allowdisplaybreaks
\newtheorem{theorem}{Theorem}[section]
\newtheorem{lemma}[theorem]{Lemma}
\newtheorem{proposition}[theorem]{Proposition}

\newtheorem{definition}[theorem]{Definition}
\newtheorem{remark}[theorem]{Remark}
\numberwithin{equation}{section}

\makeatletter
\DeclareRobustCommand*{\bfseries}{%
  \not@math@alphabet\bfseries\mathbf
  \fontseries\bfdefault\selectfont\boldmath
}
\newcommand{\reqnomode}{\tagsleft@false}

\@ifundefined{subjclassname@2020}{%
  \@namedef{subjclassname@2020}{\textup{2020} Mathematics Subject Classification}%
}{}
\makeatother

\makeatletter
\newcommand{\Xint}[1]{\mathchoice
  {\XXint\displaystyle\textstyle{#1}}%
  {\XXint\textstyle\scriptstyle{#1}}%
  {\XXint\scriptstyle\scriptscriptstyle{#1}}%
  {\XXint\scriptscriptstyle\scriptscriptstyle{#1}}\!\int}
\newcommand{\XXint}[3]{{%
  \sbox\z@{$#1{#2#3}\int$}%
  \vcenter{\hbox{$#2#3$}}\kern-.5\wd\z@
}}
\makeatother
\newcommand{\dashint}{\Xint-}

\newcommand{\mint}{\dashint}

\newcommand{\nr}[1]{\lVert #1\rVert}
\newcommand{\norm}[1]{\nr{#1}}

\makeatletter
\newcommand{\linethrough}{\mathpalette\DKMS@strike}
\newcommand{\DKMS@strike}[2]{\begingroup
  \sbox\z@{$\m@th#1#2$}%
  \rlap{\raisebox{.5\ht\z@}{\rule{\wd\z@}{.4pt}}}\box\z@
\endgroup}
\makeatother

\DeclareMathOperator{\dimH}{dim_{\mathcal H}}
\DeclareMathOperator{\loc}{loc}

\newlength{\defbaselineskip} 
\newcommand{\R}{\mathbb{R}}
\def\R{\mathbb R}

\newcommand{\eps}{\varepsilon}

\newcommand{\data}{\textnormal{\texttt{data}}}

\newcommand{\CC}{\operatorname{C}}
\newcommand{\WW}{\operatorname{W}}
\newcommand{\LL}{\mathrm{L}}

\newcommand{\rif}[1]{(\ref{#1})}
\newcommand{\tx}[1]{\textnormal{\texttt{#1}}}

\newcommand{\Ma}{\R^{N\times n}}
\newcommand{\Hosc}{\mathcal H}

\newcommand{\F}{\mathcal F}

\newcommand{\dd}{\,\mathrm d}

\newcommand{\eqn}[1]{\begingroup\def\DKMSlabel{#1}\DKMSequation}
\long\def\DKMSequation#1$$#2$${%
  \begin{equation}\label{\DKMSlabel}#2\end{equation}\endgroup}

\title[The Singular Set of Minima]{The Singular Set of Minima\\of Multiple Integrals}
\author[De Filippis]{Cristiana De Filippis}
\address{Cristiana De Filippis\\
Dipartimento SMFI, Universit\`a di Parma\\
Parco Area delle Scienze 53/a, I-43124, Parma, Italy}
\email{cristiana.defilippis@unipr.it}

\author[Kristensen]{Jan Kristensen}
\address{Jan Kristensen\\
Mathematical Institute, University of Oxford\\
Radcliffe Observatory Quarter (550)\\
Woodstock Road, Oxford, OX2 6GG}
\email{kristens@maths.ox.ac.uk}

\author[Mingione]{Giuseppe Mingione}
\address{Giuseppe Mingione\\
Dipartimento SMFI, Universit\`a di Parma\\
Parco Area delle Scienze 53/a, I-43124, Parma, Italy}
\email{giuseppe.mingione@unipr.it}

\author[Strong]{Aidan Strong}
\address{Aidan Strong\\
Mathematical Institute, University of Oxford\\
Radcliffe Observatory Quarter (550)\\
Woodstock Road, Oxford, OX2 6GG}
\email{aidan.strong@spc.ox.ac.uk}
\date{}

\begin{document}

\begin{abstract}
The Hausdorff dimension of the singular set of local minimizers of quasiconvex integrals is strictly less than the ambient dimension.
\end{abstract}

\maketitle
\setcounter{tocdepth}{1}
{\small\tableofcontents}

\section{The result}
In this paper we prove that the Hausdorff dimension of the singular set of local minimizers of quasiconvex integrals is strictly
less than the ambient dimension. The result is obtained without assuming a priori Lipschitz regularity of minima, as, in contrast, previously
done in \cite{km07}. This solves an old and classical issue.
Indeed, the problem of obtaining bounds for the Hausdorff dimension of singular sets of minimizers of quasiconvex integrals was raised several times
over the years, see for instance \cite[Section~4, p.~1079]{GiaICM}, \cite[Open problem (a), pp.~117]{GiaETH}, after the first partial regularity
result of Evans \cite{Ev,EvICM}. See \cite{df22,DLSV,DM,gm,gmjk1,gmjk2,cirving,ilr} and related references, for state-of-the-art partial regularity
results. An interesting point of the proof, whose basic building blocks and approach are already contained in \cite{km07}, is that at the end the
singular set estimate is obtained similarly to the convex case \cite{min03,min03-2,km05,km06}, that is, by proving that the gradient of
a minimizer belongs to some fractional Sobolev spaces. This, in fact, implies the desired dimension estimate. In this respect, an interesting
parallel with the explicit estimates available in the convex case emerges, as pointed out in Remark \ref{remarkino}. Remarkably, the main input
for the new arguments comes from the regularity theory of nonlocal problems, as explained in Remark \ref{remarkino2}. For surveys
of results on partial regularity and singular sets estimates we refer to \cite{min08}, \cite{GiaETH,Giusti} and related references. 

We consider general variational integrals of the type
\eqn{mainf}
$$
\F[v;\Omega] := \int_{\Omega} \! F(x,v,Dv) \, \dd x
$$
and with polynomial $p$-growth, $p\geq 2$, where $\Omega\subset \R^n$ is open, defined on Sobolev maps $v\in \WW^{1,p}(\Omega , \R^N)$, with $n,N\ge2$. 
We consider standard structural assumptions for the integrand $F$ in \eqref{mainf}. 
Here $F \colon \Omega\times\R^N\times\Ma\to\R$ is a Carathéodory integrand. 
We assume
\begin{equation*}
 \begin{gathered}
 F(x,y,\cdot)\in \CC^2(\Ma),\\
 (x,y,z)\longmapsto \partial_{zz}F(x,y,z)\quad\hbox{is continuous}
 \end{gathered}
\tag{A1}\label{smooth}
\end{equation*}
\begin{equation*}
 \nu|z|^p-L\le F(x,y,z)\le L(1+|z|^p)
 \tag{A2}\label{coercivity}
\end{equation*}
\begin{equation*}
 |\partial_{zz}F(x,y,z)|\le\Lambda(1+|z|^2)^{(p-2)/2}
 \tag{A3}
\end{equation*}
\begin{equation*}
 \begin{split}
 &\int_B\bigl[F(x_0,y_0,z_0+D\psi)-F(x_0,y_0,z_0)\bigr]\dd x\\
 &\quad\ge\lambda\int_B
   (1+|z_0|^2+|D\psi|^2)^{(p-2)/2}|D\psi|^2\dd x
 \end{split}
 \tag{A4}\label{quasiconvexity}
\end{equation*}
\begin{equation*}
 \begin{split}
 &|F(x,y,z)-F(x_0,y_0,z)|\le L_0\,\omega(|x-x_0|+|y-y_0|)(1+|z|^p),\\
 &\omega(t):=\min\{1,t^\alpha\},\qquad 0<\alpha\le1.
 \end{split}
 \tag{A5}\label{coefficients}
\end{equation*}
Here $0<\nu\le L$, $\lambda,\Lambda>0$ and $L_0\ge0$ are fixed.
The inequalities hold for all the displayed variables. In
\eqref{quasiconvexity}, $x_0\in\Omega$, $y_0\in\R^N$,
$z_0\in\Ma$, $B\subset \R^N$ is a ball and $\psi\in \CC_0^1(B , \R^N)$.
We abbreviate the relevant structural data by
$$
 \data=(n,N,p,\nu,L,\lambda,\Lambda).
$$
\begin{remark}\label{remarkino11} {\em 
The condition \textup{(A4)} is of course a suitably strong form of quasiconvexity in the sense of Morrey \cite{Morrey}.
Note that the condition \textup{(A4)} actually holds for every Sobolev test map $\psi\in \WW^{1,p}_0(B , \R^N)$.
This follows by a standard density argument and by the $p$-growth of the integrand in \textup{(A2)}.
In turn, assumption \textup{(A2)} can be weakened as described in Remark \ref{remarkino20} at the end of the paper.}
\end{remark}
\begin{definition}
A map $u\in \WW^{1,p}_{\loc}(\Omega , \R^N)$ is a local minimizer of the functional $\F$ if $\mathcal F[u;B]\le\mathcal F[u+\psi;B]$
for every ball $B\Subset\Omega$ and every
$\psi\in \WW^{1,p}_0(B , \R^N)$. 
\end{definition}
We shall widely use the vector field defined by
$$
 V(z)=(1+|z|^2)^{(p-2)/4}z,\qquad z\in\Ma.
$$
The crucial notion here is that of a regular point. 
\begin{definition}[The singular set $\Sigma_u$]
A point $x\in\Omega$ is regular if $Du$ admits a H\"older continuous
representative in a neighbourhood of $x$ contained in $\Omega$.
Writing $\operatorname{Reg}(u)$ for the set of regular points, we define
$
\Sigma_u := \Omega\setminus\operatorname{Reg}(u).
$
Thus $\operatorname{Reg}(u)$ is open and $\Sigma_u$ is relatively
closed in $\Omega$.
\end{definition}

The classical higher integrability of minimizers, based on Gehring's lemma, recalled in
\cite[Theorem~2.1]{km06}, yields that under (the only) assumption (A2) there exists a higher integrability exponent 
\eqn{higherq}
$$
q \equiv q (n,N,p,\nu, L)>p
$$
such that $Du\in \LL^q_{\loc}(\Omega ; \Ma)$. Moreover 
\eqn{gehring}
$$
 \mint_{B(x,r)} \! |Du|^q \, \dd y
 \le c\left(\mint_{B(x,2r)} \! (1+|Du|^p)\dd y\right)^{q/p}
$$
holds for every ball $B(x,2r)\Subset\Omega$. 
As in the convex case \cite{km06,min03}, the higher integrability exponent $q$ heavily intervenes in the singular set estimate. 
\begin{theorem}\label{maint}
Let $u$ be a local minimizer of the functional $\F$ in \eqref{mainf} under
\eqref{smooth}--\eqref{coefficients}. With $\alpha$ as in \textup{(A5)} and  $q$ as in \eqref{higherq}, set 
\eqn{espo}
$$
 \sigma=\min\{\alpha,q-p\},\qquad
 s_* =\min\{s_0,\sigma/2\},
$$
where $s_0>0$ is a constant depending only on $\data$ and determined during the proof. Then, for every $0<s<s_*$,
\begin{equation}
 V(Du)\in \WW^{s,2}_{\loc}(\Omega ; \Ma),\qquad
 Du\in \WW^{2s/p,p}_{\loc}(\Omega ; \Ma)
 \label{eq:mainfractional}
\end{equation}
with 
\begin{equation}
 \begin{split}
 &[V(Du)]_{s,2;B(x_0,R)}^2
       +[Du]_{2s/p,p;B(x_0,R)}^p\\
 &\qquad\le \frac{c}{R^{2s}}\left[
       \int_{B(x_0,8R)} \! (1+|Du|^p) \, \dd x
       +R^\sigma\int_{B(x_0,8R)} \! |Du|^q  \dd x\right]
 \end{split}
 \label{mainest}
\end{equation}
that holds for every ball $B(x_0,8R)\Subset\Omega$ with $0<R\le1$, 
where $c\equiv c(\textnormal{\data},L_0,\alpha,q,s)$. Moreover,
\eqn{stimaf}
$$
 \dimH\Sigma_u\le n-2s_*
$$
where $ \dimH$ denotes the usual Hausdorff dimension. 
\end{theorem}
\begin{remark}[Similarities with the convex case]\label{remarkino}{\em 
The crucial information in Theorem~\ref{maint} is the fractional
Sobolev regularity expressed by \eqref{mainest}, with \eqref{eq:mainfractional} eventually following by a standard covering argument.
Once this is established, the Hausdorff dimension estimate \eqref{stimaf} follows
by standard arguments \cite{km06,min03,min08}. We shall therefore
confine ourselves to proving \eqref{mainest}.
The exponents in \eqref{espo} reveal a close similarity with the
convex case \cite{km06}, where the estimate is
$\dimH\Sigma_u\le n-\sigma$.
Indeed, our bound reads
\[
 \dimH\Sigma_u\le n-\min\{2s_0,\sigma\},
\]
and thus recovers the convex bound whenever $2s_0\ge\sigma$.
The additional threshold $s_0$ arises from the self-improvement
argument used here to handle quasiconvex integrands. }
\end{remark}

\begin{remark}[Origins of the new arguments]\label{remarkino2} {\em 
The main new input in our proof stems from an analogy we noted with the
differentiability self-improvement discovered in \cite{KMS}
for nonlocal, possibly nonlinear equations with measurable kernels.
In the simplest model case, one considers
$u\in \WW^{s,2}(\R^n)$, with $0<s<1$, satisfying
\[
 \int_{\R^n} \! \int_{\R^n} \!
 a(x,y)
 \frac{(u(x)-u(y))(\varphi(x)-\varphi(y))}
      {|x-y|^{n+2s}}\,\dd y\dd x
 =
 \int_\Omega \! f\varphi\,\dd x
\]
for every $\varphi\in \CC^\infty_0(\Omega)$, where $a$ is measurable,
$0<\nu\le a(x,y)\le L$, and, for simplicity,
$f\in \LL^2_{\loc}(\Omega)$.
Such solutions belong to $\WW^{s+\delta,2}_{\loc}(\Omega)$ for some $\delta\in(0,1-s)$.
Such a differentiability gain does not occur in the classical local case when linear equations with measurable coefficients
are considered and it is therefore a genuinely nonlocal effect. A key feature of the nonlocal setting is the infinite measure
$|x-y|^{-n}\dd y\dd x$, whose mass diverges logarithmically
near the diagonal. For fixed $x$, polar coordinates give
\eqn{empha}
$$
 \frac{\dd y}{|x-y|^n}
 =\frac{\dd r}{r}\,\dd\theta,
 \qquad y=x+r\theta,\quad \theta\in\mathbb S^{n-1}.
$$
In \cite{KMS}, this structure is exploited through dual pairs
involving locally finite perturbations of that measure.  See  Remark \ref{remarkino3} for a more detailed comparison.
The analogy also extends to the underlying real-variable methods.
The level-set analysis in \cite{KMS} relies on
Calder\'on--Zygmund decompositions and stopping-time arguments
adapted to the diagonal and off-diagonal geometry of $\R^{2n}$.
This analysis across scales is closely related to the tools underlying Dorronsoro's characterizations
\cite{Dor,DorBesov} used here and in \cite{km07}. 
In both approaches, the control of oscillations at different
positions and scales leads to norm estimates measuring
differentiability. This connection becomes concrete in the present proof.
Building on the Caccioppoli estimates in \cite{km07}, that are here rendered in suitable versions (Lemma \ref{lem:caccioppoli}), we combine
them with Dorronsoro's estimates and integration over scales to obtain an integral estimate involving the infinite measure $\dd r/r$.
This is in Proposition \ref{prop:relative} below and in estimates \eqref{relativeZ}-\eqref{relative}. 
The estimate leads, as in \cite{KMS}, to a gain in fractional differentiability. See Section \ref{s.4} for details. Besides this improvement,
to fit the outcome into the general scheme of proof, we need to accommodate a number of statements in order to avoid any use of a priori Lipschitz
regularity. This can be done in our general setting by using some hidden  facts concerning partial regularity proofs in part already observed
in \cite{DM}. In turn, these are employed together with the previously mentioned Dorronsoro's results and a suitable use of maximal operators.
This additional analysis allows us to free all the estimates from undesired dependences on the modulus of the gradient and leads to a Carleson type estimate.}
\end{remark}

\begin{remark}{\em An example of a strongly polyconvex integral with a minimizer whose singular set has codimension $2$ has been recently found in \cite{agrt}. When dependence on $x$ is measurable, the regular set can be empty and the set of singular points can be dense, as shown in \cite{sou} directly for solutions to elliptic systems.}
\end{remark}

\section{Preliminaries and known estimates}
We shall mainly use the notation contained in \cite{km07}. In particular, we use the Euclidean norm on vectors and matrices,
while $\langle \cdot,\cdot \rangle$ denotes the standard Euclidean inner product. A constant $c$, usually denoting a quantity greater or equal than one, 
may change from line to line; its relevant dependencies are specified
in the statements. Special occurrences will be denoted by $c_1$,
$\tilde{c}$ and so on. With $x_0 \in \R^n$ and $R>0$, we denote by $B_R
\equiv B(x_0,R) :=\{x \in \R^n \ : \ |x-x_0|< R \}$ the open
(Euclidean) ball with radius $R$ and center $x_0$. Often it is
clear from the context that the balls under consideration all have
the same center and in such cases we merely write $B_R$ etc. We denote
$$
(g)_{x,r}=\mint_{B(x,r)} \! g(y) \, \dd y
$$
the integral average of $g\in \LL^1(B(x,r) ; \R^m)$ over the ball $B(x,r)$. In denoting several function spaces like $\LL^{\gamma}(\Omega), \WW^{1,\gamma}(\Omega)$,
$\gamma \geq 1$, we shall denote their vector-valued versions by $\LL^{\gamma}(\Omega ; \R^m), \WW^{1,\gamma}(\Omega ,\R^m)$ in the case the maps considered take values
in $\R^m$, $m\in \mathbb{N}$; when no confusion can arise, we shall abbreviate $\LL^{\gamma}(\Omega ; \R^m)\equiv \LL^{\gamma}(\Omega)$ and the like.
In the rest of the paper we shall often indicate global norms. In such cases, for instance, if $g\in \LL^{\gamma}(\R^N)$, we write 
$\|g\|_{\LL^{\gamma}(\R^N)}\equiv \|g\|_{\LL^{\gamma}}$.  For $0<s<1$ and
$1\le \gamma<\infty$, our convention for the fractional Gagliardo seminorm is
$$
 [g]_{s,\gamma;U}^\gamma
 =\int_U\int_U \! \frac{|g(x)-g(y)|^\gamma}{|x-y|^{n+s\gamma}} \, \dd y\dd x
$$
for $g \colon U\to \R^m$ being a measurable map and $U\subset \R^n$ an open subset. We recall that $\WW^{s,\gamma}(U;\R^m)$ is the Banach spaces made of all $L^\gamma(U;\R^m)$-regular maps $g$ such that $ [g]_{s,\gamma;U}$ is finite. Its local variant is defined in the obvious way. Note that we shall use the property
\eqn{minav}
$$
\mint_{B(x,r)} \! |g-(g)_{x,r}|^{\gamma} \, \dd y\leq 2^{\gamma}\mint_{B(x,r)} \! |g-g_0|^{\gamma} \, \dd y
$$
for every $\gamma\geq 1$ and $g_0\in \R^m$. When $\gamma=2$, it is possible to replace $2^{\gamma}$ by $1$ in the above inequality.

\subsection{Useful inequalities}
For $z_0\in\Ma$ and $t\ge0$, let
\eqn{use}
$$
 \langle z_0\rangle=(1+|z_0|^2)^{1/2},
 \qquad \phi_{z_0}(t)=\langle z_0\rangle^{p-2}t^2+t^p.
$$
The standard inequalities for $V(\cdot)$ yield
\begin{equation}
 c^{-1}\phi_{z_0}(|z-z_0|)
 \le |V(z)-V(z_0)|^2
 \le c\,\phi_{z_0}(|z-z_0|),
 \label{eq:Vequivalence}
\end{equation}
with $c\equiv c(p)$; see~\cite[Section~2]{DLSV} for every $z\in \Ma$. In particular,
\eqn{Vconsequences}
$$
 \begin{cases}
 |z-z_0|^p\le c|V(z)-V(z_0)|^2\\[4pt]
 |V(z)|^2=(1+|z|^2)^{(p-2)/2}|z|^2 \le c(1+|z|^p).
 \end{cases}
$$
The map $V(\cdot)$ is a locally Lipschitz bijection of $\Ma$ onto itself. We shall use the following excess functional
\eqn{eccessone}
$$
{\bf E}(x,\varrho):= \mint_{B(x,\varrho)} \! |V(Du)-(V(Du))_{x,\varrho}|^{2} \, \dd y
$$
whenever $u \in \WW^{1,p}(B(x,\varrho) , \R^N)$ and $B(x,\varrho) \Subset \Omega$.
We shall use it for the minimizer from Theorem \ref{maint}.
Finally, we recall that the  Hardy--Littlewood maximal operator ${\bf M}$ of a locally integrable map $g\colon \R^n \to \R^m$ is defined by
\eqn{hlm}
$$
{\bf M}(g)(x):= \sup_{r>0} \mint_{B(x,r)} \!  |g| \, \dd y\,.
$$
\subsection{Affine approximation}
For $v\in \LL^2(B(x,\varrho) , \R^N)$, let $P_{x,\varrho}v$ denote
its $\LL^2$-orthogonal projection onto the space of affine maps.
Explicitly, recalling~\cite[Section~3, p.~742]{DM},
\begin{equation}
 \begin{split}
 (P_{x,\varrho}v)(y)
   &=(v)_{x,\varrho}+Q_{x,\varrho}(v)(y-x),\\[6pt]
 Q_{x,\varrho}(v)
   &=\frac{n+2}{\varrho^2}
     \mint_{B(x,\varrho)}v(y)\otimes(y-x)\dd y.
 \end{split}
 \label{pro}
\end{equation}
Here $(a\otimes b)_{\alpha i}=a_\alpha b_i$ denotes the tensor product, with $a\in \R^N, b\in \R^n$.
In particular, $(P_{x,\varrho}v)(x)=(v)_{x,\varrho}$ and $D(P_{x,\varrho}v)=Q_{x,\varrho}(v).$ This projection is linear, reproduces
affine maps, and is bounded on normalized $\LL^{\gamma}(B(x,\varrho))$ uniformly in $x,\varrho$, for each $1\le \gamma<\infty$.
For $1\le \gamma <\infty$ and $v\in \LL^{\gamma}_{\loc}(\R^n , \R^N)$, define the $\gamma$-affine excess functional 
$$
\beta_{\gamma}(v;x,r) =\left(\mint_{B(x,2r)} \! \left|\frac{v-P_{x,2r}v}{r}\right|^{\gamma} \, \dd y\right)^{1/\gamma}
$$
and, when $v\in \LL^2_{\loc}(\R^n ,\R^N)$, the square function
\begin{equation}\label{square}
S_2v(x)=\left(\int_0^\infty \! \beta_2(v;x,r)^{2} \, \frac{\dd r}{r}\right)^{1/2}.
\end{equation}
It holds that
\begin{lemma}[Dorronsoro~\cite{Dor,DorBesov}]\label{dorro}
Let $p\ge2$ and let $v\in \WW^{1,p}(\R^n , \R^N)$ have compact
support. There exists a constant $c\equiv c(n,N,p)$ such that
\begin{equation}
 \norm{S_2v}_{\LL^2}\le c\norm{Dv}_{\LL^2},
 \qquad
 \norm{S_2v}_{\LL^p}\le c\norm{Dv}_{\LL^p}.
 \label{dorronsoro}
\end{equation}
\begin{equation}
\int_{\R^n}\int_{0}^{\infty} \! \beta_p(v;x,r)^{p} \, \frac{\dd r}{r}\dd x
 \le c\norm{Dv}_{\LL^p}^p.
 \label{besov}
\end{equation}
Moreover, if ${\bf M}$ is the Hardy--Littlewood maximal operator in \eqref{hlm}, then
\begin{equation}
 |D(P_{x,2r}v)|\le c\mint_{B(x,2r)} \! |Dv| \, \dd y
                  \le c\, {\bf M}(|Dv|)(x).
 \label{projectiongradient}
\end{equation}
\end{lemma}
\begin{proof}
Since $p\ge2$ and $v$ has compact support, $v\in \WW^{1,2}(\R^n , \R^N)$ as well.
It is easy to see that theorems \cite[Theorems~2 and~6]{Dor} hold when using the square functions defined on balls as considered in \eqref{square}. 
Thus \cite[Theorems~2 and~6]{Dor}, applied componentwise, give 
\eqn{fifi}
$$
 \norm{S_2v}_{\LL^{\gamma}}
 \le c \norm{v}_{\LL^{\gamma}}+c\norm{Dv}_{\LL^{\gamma}}, \qquad \gamma=2,p.
$$
The homogeneous bounds in \eqref{dorronsoro} are obtained by applying a standard scaling argument we briefly recall.  Let
$v_\lambda(x)=v(\lambda x)$, with $\lambda>0$.
It follows that  
$$
\begin{cases}
 \beta_2(v_\lambda;x,r)
 =\lambda\beta_2(v;\lambda x,\lambda r)\\[2pt]
 S_2v_\lambda(x)=\lambda S_2v(\lambda x).
\end{cases}
$$
Moreover, we have 
$$
\begin{cases}
 \norm{S_2v_\lambda}_{\LL^{\gamma}} =\lambda^{1-n/\gamma}\norm{S_2v}_{\LL^{\gamma}}\\[2pt]
\norm{v_\lambda}_{\LL^{\gamma}}
=\lambda^{-n/\gamma}\norm{v}_{\LL^{\gamma}}\\[2pt]
\norm{Dv_\lambda}_{\LL^{\gamma}}
 =\lambda^{1-n/\gamma}\norm{Dv}_{\LL^{\gamma}}. 
 \end{cases}
 $$
Applying  \eqref{fifi} to
$v_\lambda$, and using the above relations, and dividing by $\lambda^{1-n/\gamma}$ yields
$$
 \norm{S_2v}_{\LL^{\gamma}}
 \le c
       \lambda^{-1}\norm{v}_{\LL^{\gamma}}
       +c\norm{Dv}_{\LL^{\gamma}} , \qquad \gamma=2,p.
$$
Letting $\lambda\to\infty$ proves \eqref{dorronsoro}
for $\gamma=2,p$ (this argument can be found in \cite[Lemma 2.6]{FO}).  Similarly, \cite[Theorems~1 and~2]{DorBesov} and the embedding
$\dot \WW^{1,p}\hookrightarrow\dot B^1_{p,p}$, valid for $p\ge2$, give
\[
  \int_{\R^n}\int_0^\infty
       \beta_p(v;x,r)^p\,\frac{\dd r}{r}\,\dd x
       \le c\|v\|_{\dot B^1_{p,p}}^p
       \le c\|Dv\|_{\LL^p}^p,
\]
which proves \eqref{besov}. Here $\|\cdot\|_{\dot B^1_{p,p}}$ denotes
the homogeneous Besov seminorm, and the embedding follows from the
Littlewood--Paley characterization and $\ell^2\hookrightarrow\ell^p$. Finally, \eqref{pro} and the $\LL^1$ Poincar\'e inequality yield
\[
\begin{aligned}
 |D(P_{x,2r}v)|
 &\le c
       \mint_{B(x,2r)}\left|\frac{v-(v)_{x,2r}}{r}\right|\,\dd y \le c
       \mint_{B(x,2r)}|Dv|\,\dd y
       \le c\, {\bf M}(|Dv|)(x),
\end{aligned}
\]
which is \eqref{projectiongradient}.
\end{proof}

\subsection{Caccioppoli inequality}
The following Caccioppoli inequality appears in various forms in the literature starting with the foundational work of Evans \cite{Ev}.
The proof can be adapted from the arguments in \cite[Proposition~3.1]{km07} and
\cite[Theorem~11, pp.~3601--3602]{DLSV}; see also \cite{Giusti}, where such inequalities are called Caccioppoli inequalities of the second kind.
Here we report it in the way we need and we also include some details of the proof for the reader's convenience. 
\begin{lemma}[Weighted Caccioppoli inequality]
\label{lem:caccioppoli}
Let $u\in \WW^{1,p}_{\loc}(\Omega;\R^N)$ be a local
minimizer of the functional in \eqref{mainf} under assumptions
\textup{(A1)--(A5)}, with $p\ge2$.
Let $q>p$ be a higher integrability exponent for $Du$ as in \eqref{higherq}-\eqref{gehring}, and set
$\sigma=\min\{\alpha,q-p\}$, as in \eqref{espo}.
For every ball $B(x,2r)\Subset\Omega$, with $0<r\le1$, put
$z_0=D(P_{x,2r}u)$ and recall the definition of $\phi_{z_0}(\cdot)$ in \eqref{use}. 
Then
\begin{equation}
 \begin{split}
\mint_{B(x,r)} \! \phi_{z_0}(|Du-z_0|) \, \dd y
 &\le c\mint_{B(x,2r)} \! \phi_{z_0}\left(\frac{|u-P_{x,2r}u|}{r}\right) \, \dd y
       \\ & \quad +c_0r^\sigma\mint_{B(x,2r)} \! (1+|Du|^q) \, \dd y.
 \end{split}
 \label{eq:caccioppoli}
\end{equation}
Here $c\equiv c(\data)$ and $c_0=c_0(\data,L_0,\alpha,q)$ are independent of $z_0$
and $(u)_{x,2r}$; moreover, $c_0=0$ if $L_0=0$.
\end{lemma}

\begin{proof}
We follow the cut-off argument of \cite[Proposition~3.1]{km07} and the shifted formulation of \cite[Theorem~11]{DLSV}, keeping
track of the error terms coming from perturbation methods in the nonautonomous case. Several of the arguments in the following are
classical in partial regularity theory and we refer to the above articles, and their references, for details.
The main emphasis here is on the exact computation of the last term appearing in \rif{eq:caccioppoli} as the exponent $\sigma$ appears
in a direct way in the final Hausdorff dimension estimate \rif{stimaf}. To proceed, fix $x$ and $r$ as in the statement, write $B_\rho=B(x,\rho)$,
and set
\eqn{recallv}
$$
\begin{cases}
 a=(u)_{x,2r}\\[3pt]
 \ell(y)=P_{x,2r}u(y)=a+z_0(y-x)\\[3pt]
 v=u-\ell,
 \end{cases} \quad z_0:=\frac{n+2}{4r^2}
     \mint_{B(x,2r)} \! u(y)\otimes(y-x) \, \dd y
$$
where we have used  \rif{pro}. Throughout the proof, constants denoted by $c$ depend only on
$\data$, while constants denoted by $c_1$ may additionally depend
on $\alpha$ and $q$. Freeze the integrand $F(\cdot)$ by setting
$$
 F_0(z):=F(x,a,z),\qquad
H_0(z)=F_0(z_0+z)-F_0(z_0)-\langle DF_0(z_0),z\rangle
$$
for every $z\in \Ma$. 
Taylor's formula and \textup{(A3)} give
\begin{equation}
 |H_0(z)|\le c(\langle z_0\rangle^{p-2}|z|^2+|z|^p)
          =c\phi_{z_0}(|z|).
 \label{eq:cacc-taylor}
\end{equation}
Moreover, \textup{(A4)} and Remark \ref{remarkino11}
imply 
\eqn{ca1}
$$
 \frac 1c\int_{B_s}\phi_{z_0}(|D\xi|)\dd y
 \le\int_{B_s}H_0(D\xi)\dd y
 \qquad\text{for every }\xi\in \WW^{1,p}_0(B_s ,\R^N).
$$
For $r\le t<s\le2r$, choose $\eta\in \CC^\infty_0(B_s)$ such that $0\le\eta\le1$, $\eta=1$ on $B_t$
and $|D\eta|\lesssim 1/(s-t)$. Define
$
 \xi:=\eta v$, $\psi:=(1-\eta)v$ and $
 w:=u-\xi=\ell+\psi$, so that
$
 z_0+Dv=Du$, $z_0+D\psi=Dw$ and $
 Dv-D\psi=D\xi$. 
Therefore,
\[
\begin{aligned}
 H_0(Dv)-H_0(D\psi)
 &=
 F_0(z_0+Dv)-F_0(z_0+D\psi)
 -\langle DF_0(z_0),Dv-D\psi\rangle \\
 &=
 F_0(Du)-F_0(Dw)-\langle DF_0(z_0),D\xi\rangle.
\end{aligned}
\]
Integrating the above relation and using 
$
 \int_{B_s}D\xi\,\dd y=0$, we find   
\[
 \int_{B_s}\bigl[H_0(Dv)-H_0(D\psi)\bigr]\dd y
 =
 \int_{B_s}\bigl[F_0(Du)-F_0(Dw)\bigr]\dd y
\]
and therefore 
\eqn{ca2}
$$
\begin{aligned}
 \int_{B_s}H_0(D\xi)\,\dd y
 ={}&
 \int_{B_s}
   \bigl[H_0(D\xi)-H_0(Dv)+H_0(D\psi)\bigr]\dd y\\
 &+\int_{B_s}\bigl[F_0(Du)-F_0(Dw)\bigr]\dd y.
\end{aligned}
$$
The integrand in the first integral on the right-hand side
of \eqref{ca2} vanishes on $B_t$, since $D\xi=Dv$,
$D\psi=0$, and $H_0(0)=0$ there.
Using \eqref{eq:cacc-taylor} and the formulas for
$D\xi$ and $D\psi$, we obtain
\eqn{ca3}
$$
 \int_{B_s}\bigl[H_0(D\xi)-H_0(Dv)+H_0(D\psi)\bigr]\dd y \leq c \int_{B_s\setminus B_t}\left[\phi_{z_0}(|Dv|)
 +\phi_{z_0}\left(\frac{|v|}{s-t}\right)\right]\dd y .
$$
Minimality of $u$ gives
$$
 \int_{B_s}\bigl[F(y,u,Du)-F(y,w,Dw)\bigr]\dd y\le0.
$$
Adding and subtracting these two integrands, and using
\textup{(A5)} together with $|y-x|\le2r$, we obtain
\eqn{ca4}
$$
\begin{aligned}
 &\int_{B_s}\bigl[F_0(Du)-F_0(Dw)\bigr]\dd y\\
 &\quad\le
 \int_{B_s}\bigl[F(x,a,Du)-F(y,u,Du)\bigr]\dd y
 +\int_{B_s}\bigl[F(y,w,Dw)-F(x,a,Dw)\bigr]\dd y\\
 &\quad\le
 L_0\int_{B_s}\omega(2r+|u-a|)(1+|Du|^p)\dd y
 +L_0\int_{B_s}\omega(2r+|w-a|)(1+|Dw|^p)\dd y\\
 &\quad=:\mathcal R_u+\mathcal R_w.
\end{aligned}
$$
We now estimate $\mathcal R_u$ and $\mathcal R_w$. Recalling the definition of $z_0$ in \rif{recallv}, the identity in 
\eqref{pro} and Poincar\'e's inequality give
\begin{equation}
 |z_0|^{p+\sigma}+\mint_{B_{2r}}\left|\frac{v}{r}\right|^{p+\sigma}\dd y
 +\mint_{B_{2r}}\left|\frac{u-a}{r}\right|^{p+\sigma}\dd y
 \le c\mint_{B_{2r}}|Du|^{p+\sigma}\dd y.
 \label{affineLq}
\end{equation}
Since $\omega(\tau)\le\tau^\sigma$, Young's inequality with
conjugate exponents $1+p/\sigma$ and $1+\sigma/p$ yields 
$$
 \mathcal R_u \le c_1L_0r^\sigma\int_{B_{2r}}
       \left(1+\left|\frac{u-a}{r}\right|^{p+\sigma}+|Du|^{p+\sigma}\right)\dd y .
$$
For similar computations and more details see also \cite[pp. 349-350]{km06}. 
Using \rif{affineLq} we conclude with 
\eqn{ca5}
$$
 \mathcal R_u
  \le c_1L_0r^\sigma\int_{B_{2r}}(1+|Du|^q) \, \dd y
$$
as $p+\sigma \leq q$. 
For $\mathcal R_w$, observe that
\[
 \begin{split}
 w-a&=(1-\eta)(u-a)+\eta z_0(y-x),\\
 Dw&=(1-\eta)Du+\eta z_0-v\otimes D\eta
 \end{split}
\]
and therefore we can estimate 
\[
 \frac{|w-a|}{r}\le\frac{|u-a|}{r}+2|z_0|,
 \qquad
 |Dw|\le \frac{cr}{s-t}\left(|Du|+|z_0|+\frac{|v|}{r}\right).
\]
Applying Young's inequality once more, and then using
Poincar\'e's inequality and   \eqref{affineLq},
we find
\begin{equation}
 \begin{split}
 \mathcal R_w
 &\le c_1L_0r^\sigma\left(\frac{r}{s-t}\right)^p\int_{B_{2r}}
       \left[1+\left|\frac{u-a}{r}\right|^{p+\sigma}+|z_0|^{p+\sigma}
                   +|Du|^{p+\sigma}+\left|\frac{v}{r}\right|^{p+\sigma}\right]\dd y\\
 &\le c_1L_0r^\sigma\left(\frac{r}{s-t}\right)^p
       \int_{B_{2r}}(1+|Du|^q)\dd y.
 \end{split}
 \label{eq:cacc-competitor-error}
\end{equation}
To conclude, set
\[
 \begin{aligned}
 H(\varrho)&=\int_{B_\varrho}\phi_{z_0}(|Dv|)\dd y
 \end{aligned}
\]
Using $\phi_{z_0}(\kappa b)\le\kappa^p\phi_{z_0}(b)$ for
$\kappa\ge1$, connecting estimates from \rif{ca1} to \rif{eq:cacc-competitor-error}, we find
\[
 H(t)\le c_*[H(s)-H(t)]
       +\left(\frac{r}{s-t}\right)^p
        \left[c\int_{B_{2r}}\phi_{z_0}\left(\frac{|v|}{r}\right)\dd y+c_1L_0r^\sigma\int_{B_{2r}}(1+|Du|^q)\dd y\right]. 
\]
Here $c_*=c_*(\data)$ and $c=c(\data)$, while
$c_1=c_1(\data,\alpha,q)$.
Filling the hole -- see \cite{Widman} -- that is, adding $c_*H(t)$ to both sides and dividing by $c_*+1$, for  $\theta\equiv \theta(\data):= c_*/(c_*+1)\in(0,1)$, we obtain
\begin{equation}
 H(t)\le\theta H(s)
       +\left(\frac{r}{s-t}\right)^p
       \left[c\int_{B_{2r}}\phi_{z_0}\left(\frac{|v|}{r}\right)\dd y+c_1L_0r^\sigma\int_{B_{2r}}(1+|Du|^q)\dd y\right].
 \label{eq:cacc-hole-filling}
\end{equation}
Applying Lemma \ref{iterlem} below, and recalling the definition of $v$ in \rif{recallv}, we obtain
$$
 \int_{B_r}\phi_{z_0}(|Du-z_0|) \, \dd y
 \le c\int_{B_{2r}}\phi_{z_0}
                \left(\frac{|u-P_{x,2r}u|}{r}\right)\dd y +c_1L_0r^\sigma\int_{B_{2r}}(1+|Du|^q) \, \dd y.
$$
Dividing by $|B_r|$ proves \eqref{eq:caccioppoli}.
\end{proof}

\begin{lemma}[\hspace{-0.2mm}\cite{GG82}]\label{iterlem}
Let $r>0$, $p\ge0$, $0\le\theta<1$, and
$\tx{a},\tx{b}\ge0$.
Let $H\colon[r,2r]\to[0,\infty)$ be a bounded function.
Assume that
\[
 H(t)\le \theta H(s)
       +\frac{\tx{a}}{(s-t)^p}+\tx{b}
 \qquad\text{whenever }r\le t<s\le2r.
\]
Then
\[
 H(r)\le c(p,\theta)
       \left(\frac{\tx{a}}{r^p}+\tx{b}\right).
\]
\end{lemma}

\section{A Carleson type estimate for the excess}\label{c.s}
This is in the following:
\begin{proposition}\label{prop:relative}
Under the assumptions of Theorem~\ref{maint}, 
there are constants $K=K(\data)\ge1$ and
$M=M(\data,L_0,\alpha,q)\ge0$ such that, for every
$B(x_0,4R)\Subset\Omega$, $0<R\le1$, and $z_0\in\Ma$,
\begin{equation}
 \begin{split}
 &\int_{B(x_0,R)}\int_0^R \! {\bf E}(x,r)\frac{\dd r}{r} \, \dd x\\
 &\quad\le K\int_{B(x_0,4R)} \! |V(Du)-V(z_0)|^2 \, \dd x
          +MR^\sigma\int_{B(x_0,4R)} \! (1+|Du|^q) \, \dd x.
 \end{split}
 \label{relativeZ}
\end{equation}
In particular, 
\begin{equation}
 \begin{split}
 &\int_{B(x_0,R)}\int_0^R \! {\bf E}(x,r) \, \frac{\dd r}{r}\dd x\\
 &\quad\le K|B(x_0,4R)|{\bf E}(x_0,4R)
          +MR^\sigma\int_{B(x_0,4R)} \! (1+|Du|^q) \, \dd x.
 \end{split}
 \label{relative}
\end{equation}
If $L_0=0$, then $M=0$. The functional ${\bf E}(x,r)$ has been defined in \eqref{eccessone}. 
\end{proposition}
\begin{proof} Throughout the proof, norms without a specified domain
are taken over $\R^n$.
Fix $z_0\in\Ma$, write $w(y)=u(y)-z_0y$, and choose
$\eta\in C_0^\infty(B(x_0,4R))$ such that $\eta=1$ on
$B(x_0,3R)$ and $|D\eta|\lesssim 1/R$. Finally, set $v=\eta(w-(w)_{x_0,4R})$ and extend it by zero
to the whole of $\R^n$. Poincar\'e's inequality gives, for
$\gamma=2,p$,
\begin{equation}
 \norm{Dv}_{\LL^\gamma}^\gamma
 \le c\int_{B(x_0,4R)} \! |Du-z_0|^\gamma \, \dd y.
 \label{cutoff}
\end{equation}
For $x\in B(x_0,R)$ and $0<r<R$, the ball $B(x,2r)$ lies in
$B(x_0,3R)$. Set $A_{x,r}=D(P_{x,2r}u)$.
Linearity of the projection and reproduction of affine maps give
\eqn{eq:affineresidual}
$$
 u-P_{x,2r}u=v-P_{x,2r}v, \qquad A_{x,r}=z_0+D(P_{x,2r}v).
$$
Let
$$
 \varepsilon(x,r):=r^\sigma\mint_{B(x,2r)} \! (1+|Du|^q) \, \dd y. 
$$
Fubini implies
\eqn{error}
$$
\int_{B(x_0,R)}\int_0^R  \! \varepsilon(x,r) \, \frac{\dd r}{r}\dd x\le\frac{R^\sigma}{\sigma}
                      \int_{B(x_0,3R)} \! (1+|Du|^q) \, \dd y.
$$
Moreover, when $p>2$, by \eqref{projectiongradient} we have 
\eqn{buck}
$$
 \langle A_{x,r}\rangle^{p-2}
 \le c \langle z_0\rangle^{p-2}+c\, {\bf M}(|Dv|)(x)^{p-2} .
$$
Using the minimizing property of the mean \eqref{minav},
\eqref{eq:Vequivalence}, \eqref{eq:caccioppoli}, and \eqref{buck} when $p>2$, we obtain
\[
\begin{aligned}
 {\bf E}(x,r)
 &\le \mint_{B(x,r)} \! |V(Du)-V(A_{x,r})|^2 \, \dd y \\
       &\le c\mint_{B(x,r)} \! \phi_{A_{x,r}}(|Du-A_{x,r}|) \, \dd y\\
 &\le c\mint_{B(x,2r)} \! \phi_{A_{x,r}} \left(\frac{|u-P_{x,2r}u|}{r}\right) \, \dd y
       +c_0\varepsilon(x,r)\\
 &=c\langle A_{x,r}\rangle^{p-2}\beta_2(v;x,r)^2
       +c\,\beta_p(v;x,r)^p+c_0\varepsilon(x,r)\\
 &\le c\bigl(\langle z_0\rangle^{p-2}
       +{\bf M}(|Dv|)(x)^{p-2}\bigr)\beta_2(v;x,r)^2 +c\,\beta_p(v;x,r)^p+c_0\varepsilon(x,r).
\end{aligned}
\]
When $p>2$, we can integrate over $x$ and $r$ and apply
H\"older's inequality with conjugate exponents $p/(p-2)$ and $p/2$.
The estimates \eqref{dorronsoro}--\eqref{besov} and the
boundedness of ${\bf M}$ on $\LL^p$ give
 \begin{flalign*}
 &\int_{B(x_0,R)}\int_0^R \! {\bf E}(x,r)\frac{\dd r}{r} \, \dd x\\
 &\quad\le c\langle z_0\rangle^{p-2}\norm{S_2v}_{\LL^2}^2
  +c\norm{ {\bf M}(|Dv|)}_{\LL^p}^{p-2}\norm{S_2v}_{\LL^p}^2\\ &\qquad \quad +c\norm{Dv}_{\LL^p}^p+c_0\int_{B(x_0,R)}\int_0^R \! \varepsilon(x,r) \, \frac{\dd r}{r}\dd x\\
 &\quad\le c\langle {z_0}\rangle^{p-2}\norm{Dv}_{\LL^2}^2
                         +c\norm{Dv}_{\LL^p}^p +c_0\int_{B(x_0,R)}\int_0^R \! \varepsilon(x,r) \, \frac{\dd r}{r}\dd x\\
 &\quad\le c\int_{B(x_0,4R)} \! \phi_{z_0}(|Du-z_0|) \, \dd y+c_0\int_{B(x_0,R)}\int_0^R \! \varepsilon(x,r) \, \frac{\dd r}{r}\dd x.
 \end{flalign*}
In the last line we have used \eqref{cutoff}; use of 
\eqref{eq:Vequivalence} and \eqref{error} in the above display proves 
\eqref{relativeZ}. Recalling that $V(\cdot)$ is bijective, choosing $z_0=V^{-1}((V(Du))_{x_0,4R})$  proves
\eqref{relative}. The case $p=2$ follows by the same argument, using only
the $\LL^2$-estimate in \eqref{dorronsoro}.
\end{proof}

\section{An elementary improvement and the proof of Theorem \ref{maint} }\label{s.4}
We state an elementary inequality for general parameters $\sigma, K, M$. 

\begin{lemma}\label{selfimprovement}
Let $g\in \LL^2_{\loc}(\Omega ; \R^m)$ and
$\mu\in \LL^1_{\loc}(\Omega)$, $\mu\ge0$. Define
$$
 \Hosc_g(x,r)=\mint_{B(x,r)} \! |g-(g)_{x,r}|^2 \, \dd y.
$$
Suppose that
\begin{equation}
\int_{B(x_0,R)}\int_0^R \! \Hosc_g(x,r) \, \frac{\dd r}{r}\dd x\le K|B(x_0,4R)| \Hosc_g(x_0,4R)
           +MR^\sigma\int_{B(x_0,4R)} \! \mu \, \dd x
 \label{abstractrelative}
\end{equation}
holds whenever $B(x_0,4R)\Subset\Omega$ and $0<R\le1$, with fixed
$K\ge1$, $M\ge0$ and $\sigma>0$. Set
\begin{equation}
 s_0=\frac{1}{2\log16}
       \log\!\left(1+\frac{\log4}{4^nK}\right).
 \label{eq:exponent}
\end{equation}
Then $g\in \WW^{s,2}_{\loc}(\Omega ; \R^m)$ for every
$0<s<\min\{s_0,\sigma/2\}$, and
\begin{equation}
 \begin{split}
 [g]_{s,2;B(x_0,R)}^2
 &\le \frac{c}{R^{2s}}\left[\int_{B(x_0,8R)} \! |g|^2 \, \dd x
          +MR^\sigma\int_{B(x_0,8R)} \! \mu \, \dd x\right]
 \end{split}
 \label{eq:abstractfractional}
\end{equation}
for $B(x_0,8R)\Subset\Omega$, $0<R\le1$, where
$c\equiv c(n,K,\sigma,s)$. 
\end{lemma}

\begin{proof}
\noindent\emph{Step 1.}
In the following we denote $B_t\equiv B(0,t)$. 
We first work on $B_8=B(0,8)\Subset\Omega$ and prove \eqref{eq:abstractfractional}
on $B_1$. For $0<r,t\le1$, define
$$
{\bf E}(r):=\int_{B_{1+r}} \! \Hosc_g(x,r) \, \dd x,
 \qquad
{\bf I}(t):=\int_0^t \! {\bf E}(r) \, \frac{\dd r}{r},
$$
initially allowing ${\bf I}(t)=\infty$. Set
\[
 \mathcal B:=M\int_{B_8}\! \mu \, \dd x,
 \qquad
 \mathcal A:=\int_{B_8} \! |g|^2 \, \dd x+\mathcal B.
\]
Fix $0<t\le1/4$ and integrate \eqref{abstractrelative},
with radius $R\equiv t$, over centres $z\in B_{1+2t}$.
By Tonelli's theorem 
\[
\begin{aligned}
 &\int_{B_{1+2t}}\int_{B(z,t)}\int_0^t
      \! \Hosc_g(x,r) \, \frac{\dd r}{r}\dd x\dd z\\
 &\quad=
 \int_{B_{1+2t}}\int_{B_{1+3t}}\int_0^t
     \!  \mathbf{1}_{\{|x-z|<t\}}\Hosc_g(x,r)
       \, \frac{\dd r}{r}\dd x\dd z\\
 &\quad=
 \int_0^t\int_{B_{1+3t}}
      \! \Hosc_g(x,r)
       \left(
          \int_{B_{1+2t}}
             \! \mathbf{1}_{\{|x-z|<t\}} \, \dd z
       \right)\dd x\frac{\dd r}{r}\\
       &\quad =
 \int_0^t\int_{B_{1+3t}} \!
       |B(x,t)\cap B_{1+2t}|\,
       \Hosc_g(x,r) \, \dd x\frac{\dd r}{r}.
\end{aligned}
\]
If $0<r\le t$ and $x\in B_{1+r}$, then
$B(x,t)\subset B_{1+2t}$. Hence the last expression is
bounded from below by
\[
 |B_t|\int_0^t\int_{B_{1+r}}
     \!  \Hosc_g(x,r) \, \dd x\frac{\dd r}{r}
 =|B_t|{\bf I}(t).
\]
For the oscillation term (first term) on the right-hand side of \eqref{abstractrelative}, we have
\[
 K|B_{4t}|\int_{B_{1+2t}} \! \Hosc_g(z,4t) \, \dd z
 \le K|B_{4t}|{\bf E}(4t).
\]
For the error term (second term) on the right-hand side of \eqref{abstractrelative}, we have 
\[
 Mt^\sigma\int_{B_{1+2t}}\int_{B(z,4t)}
                   \! \mu(y) \, \dd y\dd z
=Mt^\sigma\int_{B_{1+6t}} \! \mu(y)|B(y,4t)\cap B_{1+2t}| \, \dd y\le |B_{4t}|\mathcal B t^\sigma.
\]
Combining these estimates and dividing by $|B_t|$ yields
\begin{equation}
{\bf I}(t)\le 4^nK{\bf E}(4t)+4^n\mathcal B t^\sigma,
 \qquad 0<t\le1/4.
 \label{eq:tailinequality}
\end{equation}
Using \eqref{minav} (and the subsequent observation for $\gamma=2$) we have 
$$
{\bf E}(r)\le \int_{B_{1+r}}\mint_{B(x,r)} \! |g(y)|^2 \, \dd y\dd x \le  \int_{B_3} \! |g|^2 \, \dd x,
 \qquad 0<r\le1. 
$$
Taking $t=1/4$ in \eqref{eq:tailinequality}, and using the above inequality, yields 
$$
{\bf I}(1/4)\le 4^nK \int_{B_3} \! |g|^2 \, \dd x+4^n\mathcal B.
$$
Consequently,
\begin{flalign}
\nonumber 
{\bf I}(1) & ={\bf I}(1/4)+\int_{1/4}^1 \! {\bf E}(r) \, \frac{\dd r}{r}\\
& \le (4^nK+\log4)\int_{B_3} \! |g|^2 \, \dd y+4^n\mathcal B \le c(n,K)\mathcal A.\label{eq:tailfinite}
\end{flalign}
In particular, this shows that ${\bf I}(1)$ is finite. 

\medskip
\noindent\emph{Step 2: Gehring type effect and fractional differentiability.} With $s_0$ as in \eqref{eq:exponent}, 
fix $0<s<\min\{s_0,\sigma/2\}$ and, 
for $0<\delta<1/4$, define
\eqn{wecan}
$$
{\bf J}(\delta)
 :=\int_0^1 \! \min\{r^{-2s},\delta^{-2s}\}\,{\bf E}(r) \, \frac{\dd r}{r}.
$$
By \eqref{eq:tailfinite} we have that ${\bf J}(\delta)$ is finite as 
$
{\bf J}(\delta)\le\delta^{-2s}{\bf I}(1)<\infty.
$
Note that 
\[
 \min\{r^{-2s},\delta^{-2s}\}
 =1+2s\int_{\max\{r,\delta\}}^1 \! t^{-2s} \, \frac{\dd t}{t}.
\]
Substituting this identity into the definition of ${\bf J}(\delta)$ in \eqref{wecan}, 
and applying Tonelli's theorem, we find
\begin{equation*}
 \begin{split}
{\bf J}(\delta)
 &=\int_0^1 \! {\bf E}(r)\, \frac{\dd r}{r}
   +2s\int_0^1\int_\delta^1
       \! \mathbf{1}_{\{r<t\}}\,t^{-2s-1}{\bf E}(r) \, \dd t
      \frac{\dd r}{r}\\
 &=\int_0^1 \! {\bf E}(r) \, \frac{\dd r}{r}
   +2s\int_\delta^1t^{-2s-1}\int_0^1
      \! \mathbf{1}_{\{r<t\}}\,{\bf E}(r) \, \frac{\dd r}{r} 
       \dd t\\
 &={\bf I}(1)
   +2s\int_\delta^1 \! \frac{{\bf I}(t)}{t^{2s}} \, \frac{\dd t}{t}\\
 &={\bf I}(1) +2s\int_\delta^{1/4} \! \frac{{\bf I}(t)}{t^{2s}} \, \frac{\dd t}{t}+2s\int_{1/4}^1 \! \frac{{\bf I}(t)}{t^{2s}} \, \frac{\dd t}{t}\\
    &\leq{\bf I}(1)+2s\int_\delta^{1/4} \! \frac{{\bf I}(t)}{t^{2s}} \, \frac{\dd t}{t}+ (4^{2s}-1){\bf I}(1)\\
   &= 4^{2s}{\bf I}(1)
   +2s\int_\delta^{1/4} \! \frac{{\bf I}(t)}{t^{2s}} \, \frac{\dd t}{t}.
 \end{split}
\end{equation*}
In order to estimate the last integral in the above display we apply \eqref{eq:tailinequality}.
Since $2s<\sigma$, it follows that
\begin{equation}
 \begin{split}
{\bf J}(\delta)
 &\le 4^{2s}{\bf I}(1)
      +2s4^nK\int_\delta^{1/4} \! \frac{{\bf E}(4t)}{t^{2s}} \, \frac{\dd t}{t}
      +2s4^n \int_\delta^{1/4} \! t^{\sigma-2s} \, \frac{\dd t}{t} \mathcal B\\
 &\le 4^{2s}{\bf I}(1)
      +2s4^{n+2s}K\int_{4\delta}^1 \! \frac{{\bf E}(r)}{r^{2s}} \, \frac{\dd r}{r}
      +\frac{2s4^{n+2s-\sigma}}{\sigma-2s}\mathcal B\\
 &\le 4^{2s}{\bf I}(1)
      +2s4^{n+2s}K{\bf J}(\delta)
      +\frac{2s4^{n+2s-\sigma}}{\sigma-2s}\mathcal B.
 \end{split}
 \label{eq:weightedabsorption}
\end{equation}
Here we used that the truncated weight equals $r^{-2s}$ for
$r\ge4\delta$, so that
\[
\int_{4\delta}^1 \! \frac{{\bf E}(r)}{r^{2s}} \, \frac{\dd r}{r}
=
\int_{4\delta}^1 \! \min\{r^{-2s},\delta^{-2s}\}\,{\bf E}(r) \, \frac{\dd r}{r}
\le {\bf J}(\delta).
\]
Setting 
$h:=2s_0\log4>0$, and recalling \eqref{eq:exponent}, 
we obtain
\[
 2s_04^nK4^{2s_0}
 =\frac{he^h}{e^{2h}-1}
 =\frac{h}{e^h-e^{-h}}<\frac12.
\]
Consequently, for $s \leq s_0$, we find 
$
 2s4^nK4^{2s_0}<1/2 
$
and we can therefore absorb the second term
on the right-hand side of \eqref{eq:weightedabsorption}.
Together with \eqref{eq:tailfinite}, this yields
$
{\bf J}(\delta)\le c(n,K,\sigma,s)\mathcal A,
$
uniformly in $\delta$. Letting $\delta\downarrow0$ we conclude that
\begin{equation}
 \int_0^1 \! \frac{{\bf E}(r)}{r^{2s}} \, \frac{\dd r}{r}
 \le c(n,K,\sigma,s)\mathcal A.
 \label{weightedoscillation}
\end{equation}
The standard Littlewood--Paley characterization of
$\WW^{s,2}=B^s_{2,2}$, in its local mean-oscillation form, now yields
\eqn{unitfractional}
$$
 [g]_{s,2;B_1}^2
 \le c(n,s)
       \norm{g}_{\LL^2(B_2)}^2
       + c(n,s)\int_0^1 \! \frac{{\bf E}(r)}{r^{2s}} \, \frac{\dd r}{r}\leq c\mathcal A
$$
where $c\equiv c(n,K,\sigma,s)$; see~\cite[Theorems~1 and~2]{DorBesov}.
See also Remark \ref{remarkino4} below for a self-contained path leading from \eqref{weightedoscillation} to \eqref{unitfractional}. 
\begin{remark}{\em 
If $M=0$, then $\mathcal B=0$ and the error term in
\eqref{eq:weightedabsorption} is absent. The restriction
$2s<\sigma$ is then unnecessary, so the same argument applies
to every $0<s<s_0$, with constant $c(n,K,s)$.}
\end{remark}

\medskip
\noindent\emph{Step 3: Rescaling.}
Let $B(x_0,8R)\Subset\Omega$, with $0<R\le1$, and define
\[
 \Omega_R:=\frac{\Omega-x_0}{R},
 \qquad
 g_R(y):=g(x_0+Ry),
 \qquad
 \mu_R(y):=\mu(x_0+Ry).
\]
Then $B_8\Subset\Omega_R$ and $
 \Hosc_{g_R}(z,\varrho)
 =\Hosc_g(x_0+Rz,R\varrho).
$
For every $B(z_0,4\varrho)\Subset\Omega_R$ with
$0<\varrho\le1$, apply \eqref{abstractrelative} at centre
$x_0+Rz_0$ and radius $R\varrho$. This is admissible because
$R\varrho\le1$. Changing variables and dividing by $R^n$ gives
\[
\int_{B(z_0,\varrho)}\int_0^\varrho
       \! \Hosc_{g_R}(z,t) \, \frac{\dd t}{t}\dd z\\
\le
 K|B_{4\varrho}|\Hosc_{g_R}(z_0,4\varrho)
 +(MR^\sigma)\varrho^\sigma
       \int_{B(z_0,4\varrho)} \! \mu_R \, \dd z.
\]
Thus the rescaled functions satisfy the same hypothesis,
with $K$ and $\sigma$ unchanged and with $M_R=MR^\sigma$.
Applying \eqref{unitfractional}, we find
$$
 [g_R]_{s,2;B_1}^2
 \le c  \int_{B_8} \! |g_R|^2 \, \dd y
       +cMR^\sigma\int_{B_8} \! \mu_R \, \dd y, 
$$
where $c\equiv c(n,K,\sigma,s)$. 
Scaling back this inequality proves \eqref{eq:abstractfractional} and the asserted
local fractional regularity. 
\end{proof}
\begin{proof}[Proof of Theorem \ref{maint}]
As observed in Remark \ref{remarkino}, it is sufficient to prove \eqref{mainest}. By \eqref{Vconsequences}, $V(Du)\in \LL^2_{\loc}(\Omega;\Ma)$, and
by \eqref{gehring}, $Du\in \LL^q_{\loc}(\Omega;\Ma)$. Apply
Lemma~\ref{selfimprovement} to Proposition~\ref{prop:relative}
with $g=V(Du)$ and $\mu=1+|Du|^q$. For every $0<s<s_*$,
\begin{equation*}
 \begin{split}
 &[V(Du)]_{s,2;B(x_0,R)}^2\\
 &\quad\le \frac{c}{R^{2s}}\left[
    \int_{B(x_0,8R)}(1+|Du|^p)\dd x
    +R^\sigma\int_{B(x_0,8R)} \! |Du|^q \, \dd x\right],
 \end{split}
\end{equation*}
with $s_*$ as in \eqref{espo}, $s_0$ as in
\eqref{eq:exponent}, and $K=K(\data)$. Moreover,
\eqref{Vconsequences}$_1$ implies
$$
 [Du]_{2s/p,p;B}^p
 =\int_B\int_B \!
     \frac{|Du(x)-Du(y)|^p}{|x-y|^{n+2s}} \, \dd y\dd x \le c[V(Du)]_{s,2;B}^2
$$
whenever $B\subset \Omega$ is a ball. 
This proves \eqref{eq:mainfractional} and \eqref{mainest} and concludes the proof of Theorem \ref{maint} according to what is written in Remark \ref{remarkino}.
\end{proof}

\begin{remark}[Comparison with the self-improvement mechanism
of \cite{KMS}]\label{remarkino3}
{\em
The common starting point between this paper and \cite{KMS} is the finiteness of an energy
with respect to a measure whose density is singular along
the diagonal and whose radial component is $\dd r/r$, as in \rif{empha}. This already encodes, in a sense, some additional regularity. 
Reverse estimates with respect to this singular measure eventually yield control of an energy density carrying
an additional singular power. This produces higher
differentiability: in \cite{KMS} through higher integrability
of the dual quantity, and here through a direct improvement
of the scale weight. In \cite{KMS}, fix a fractional order $\gamma\in(0,1)$ and take $\varepsilon>0$ suitably small.
The dual pair is then defined as 
\[
 \mathcal U_\varepsilon(x,y)
 :=\frac{|u(x)-u(y)|}{|x-y|^{\gamma+\varepsilon}},
 \qquad
 \dd\mu_\varepsilon(x,y):=\frac{\dd x\dd y}{|x-y|^{n-2\eps}}.
\]
The less singular measure compensates exactly for the
additional power in the quotient at exponent $2$ in the sense that 
$
\nr{ \mathcal U_\varepsilon}_{\LL^2(\R^{2n} , \mu_\eps)}^2= [u]_{\gamma,2;\R^n}^2.
$
Higher integrability,
$\mathcal U_\varepsilon\in \LL^{2+\delta}_{\loc}(\R^{2n} , \mu_\varepsilon)$,
then leaves an additional singular factor
$|x-y|^{-\varepsilon\delta}$, since
\[
 \mathcal U_\varepsilon^{2+\delta}\dd\mu_\varepsilon
 =\frac{|u(x)-u(y)|^{2+\delta}}
 {|x-y|^{n+(2+\delta)\gamma+\varepsilon\delta}} \, \dd x\dd y.
\]
The resulting differentiability order is
$\gamma+\varepsilon\delta/(2+\delta)$, i.e., $u \in \WW^{\gamma+\varepsilon\delta/(2+\delta), 2+\delta}_{\loc}$.
In the quasiconvex setting, the reverse estimate \eqref{eq:tailinequality} yields the following
improvement via a singularity gain in the denominator 
\[
\int_0^1 \! {\bf E}(r) \, \frac{\dd r}{r}<\infty \Longrightarrow  \int_0^1 \! \frac{{\bf E}(r)}{r^{2s}} \, \frac{\dd r}{r}
<\infty
\]
for suitably small $s>0$, provided a reverse inequality is satisfied.}
\end{remark}

\begin{remark}\label{remarkino4} {\em For the reader's sake we report a short, standard proof of the fact that \eqref{weightedoscillation}
    implies \eqref{unitfractional}, which is again based on certain orthogonality relations. Note that 
\eqn{varia}
$$
 \Hosc_g(x,r)
 =\frac{1}{2}\, 
   \mint_{B(x,r)}\mint_{B(x,r)} \! |g(y)-g(z)|^2 \, \dd y\dd z.
$$
Indeed, using the identity
$$
 \mint_{B(x,r)}\mint_{B(x,r)} \! \langle  g(y),g(z)\rangle \, \dd y\dd z= |(g)_{x,r}|^2,
$$
and expanding the squares, we obtain
 \[
\begin{aligned}
 \frac12\mint_{B(x,r)}\mint_{B(x,r)} |g(y)-g(z)|^2\dd y\dd z
 &=\mint_{B(x,r)}|g|^2\dd y-|(g)_{x,r}|^2\\
 &=\mint_{B(x,r)}|g-(g)_{x,r}|^2\dd y
 =\Hosc_g(x,r).
\end{aligned}
\]
Using \rif{varia} we now have 
\eqn{varia2}
$$
 \begin{aligned}
& \int_0^1\frac{{\bf E}(r)}{r^{2s}}\frac{\dd r}{r}\\
 &\quad =
 \frac1{2\omega_n^2}
 \int_0^1\frac1{r^{2n+2s+1}}
 \int_{B_{1+r}}
 \int_{B(x,r)}\int_{B(x,r)}
 |g(y)-g(z)|^2\dd y\dd z\dd x\dd r,
 \end{aligned}
$$
where $\omega_n$ is the measure of the unit ball, i.e., $\omega_n:=|B(0,1)|$. For $y,z\in B_1$, write $d=|y-z|$. If $0<d\le1$ and
$3d/4\le r\le d$, then
$
 B((y+z)/2, d/4)
 \subset B_{1+r}\cap B(y,r)\cap B(z,r).
$
Thus, using \rif{varia2} and Tonelli's theorem,
\begin{flalign*}
&  \iint_{\substack{y,z\in B_1\\0<|y-z|\le1}}
 \frac{|g(y)-g(z)|^2}{|y-z|^{n+2s}}\dd y\dd z\\
 & \quad \leq c  \iint_{\substack{y,z\in B_1\\0<|y-z|\le1}}
 |g(y)-g(z)|^2\,
 \left(\int_{3d/4}^d \frac{\dd r}{r^{2n+2s+1}}\right)|y-z|^n\dd y\dd z \leq c \int_0^1\frac{{\bf E}(r)}{r^{2s}}\frac{\dd r}{r}
 \end{flalign*}
where $c\equiv c (n,s)$. 
On the other hand, we trivially have that 
\[
 \iint_{\substack{y,z\in B_1\\|y-z|>1}}
 \frac{|g(y)-g(z)|^2}{|y-z|^{n+2s}}\dd y\dd z
 \le 4\omega_n\int_{B_1}|g|^2\dd y.
\]
Matching the estimates in the last two displays yields the first inequality in \eqref{unitfractional}.}
\end{remark}

\begin{remark}[A weaker coercivity assumption]
\label{remarkino20}
{\em Assumption \textup{(A2)} can be replaced by
\eqn{hamh}
$$
 |F(x,y,z)|\le L(1+|z|^p),
 \qquad F(x,y,z)\ge\Psi(z),
$$
where $\Psi\colon\Ma\to\R$ is continuous and, for some $\nu>0$,
\[
 \int_B\bigl[\Psi(D\varphi)-\Psi(0)\bigr]\dd x
 \ge \nu\int_B|D\varphi|^p\dd x
\]
for every ball $B$ and every
$\varphi\in \CC_0^1(B,\R^N)$, while retaining
\textup{(A1)} and \textup{(A3)--(A5)}.
Increasing $L$ if necessary, we may also assume
$|\Psi(0)|\le L$. Partial regularity of minima holds under such a replacement as for instance shown in \cite{Hamburger96}.
Indeed, the lower bound in \rif{hamh} is precisely
\cite[Hypothesis~H4, p.~260]{Hamburger96}, with
$\widetilde F=\Psi$.
Moreover, \textup{(A1)} and \textup{(A3)} imply
Hypothesis~H1 of that paper.
Therefore, \cite[Theorem~4.2, p.~272]{Hamburger96}
yields the higher integrability in \eqref{higherq}
and the estimate \eqref{gehring}.
Related formulations in terms of mean coercivity and strong
quasiconvexity are considered, in the autonomous $(p,q)$-growth setting,
in \cite{chjk,gmjk1, gmjk2}. The pointwise lower bound in \textup{(A2)} is used in our
proof only to obtain this preliminary higher integrability.
Consequently, all conclusions of Theorem~\ref{maint}
remain valid under these alternative assumptions.}
\end{remark}

\hspace{1cm}

{\bf Acknowledgments and use of AI.} This work is supported by the European Research Council, through the ERC StG project NEW, nr.~101220121.
The topic of this paper was indeed a main one in the research proposal of the project NEW, presented in October 2024.
In the writing of the paper the authors benefited from the use of GPT 6 Astra to review, streamline and edit the proofs.
The mathematical ideas belong to the authors.


\begin{thebibliography}{99}



\bibitem{chjk}
\textsc{C.-Y.~Chen \& J.~Kristensen},
\textit{On coercive variational integrals},
Nonlinear Analysis: Theory, Methods \& Applications \textbf{153} (2017), 213--229,
\href{https://doi.org/10.1016/j.na.2016.09.011}
{doi:10.1016/j.na.2016.09.011}.
       
\bibitem{df22}
\textsc{C.~De Filippis},
\textit{Quasiconvexity and partial regularity via nonlinear potentials},
J.~Math.~Pures Appl. (9) \textbf{163} (2022), 11--82.
\href{https://doi.org/10.1016/j.matpur.2022.05.001}
{doi:10.1016/j.matpur.2022.05.001}.

\bibitem{DLSV}
\textsc{L.~Diening, D.~Lengeler, B.~Stroffolini \& A.~Verde},
\textit{Partial regularity for minimizers of quasi-convex functionals
with general growth},
SIAM J.~Math.~Anal. \textbf{44} (2012), 3594--3616.
\href{https://doi.org/10.1137/120870554}
{doi:10.1137/120870554}.

\bibitem{Dor}
\textsc{J.R.~Dorronsoro},
\textit{A characterization of potential spaces},
Proc.~Amer.~Math.~Soc. \textbf{95} (1985), 21--31.
\href{https://doi.org/10.1090/S0002-9939-1985-0796440-3}
{doi:10.1090/S0002-9939-1985-0796440-3}.

\bibitem{DorBesov}
\textsc{J.R.~Dorronsoro},
\textit{Mean oscillation and Besov spaces},
Canad.~Math.~Bull. \textbf{28} (1985), 474--480.
\href{https://doi.org/10.4153/CMB-1985-058-3}
{doi:10.4153/CMB-1985-058-3}.

\bibitem{DM}
\textsc{F.~Duzaar \& G.~Mingione},
\textit{Regularity for degenerate elliptic problems via
$p$-harmonic approximation},
Ann.~Inst.~H.~Poincar\'e Anal. Non Lin\'eaire
\textbf{21} (2004), 735--766.
\href{https://doi.org/10.1016/j.anihpc.2003.09.003}
{doi:10.1016/j.anihpc.2003.09.003}.

\bibitem{Ev}
\textsc{L.C.~Evans},
\textit{Quasiconvexity and partial regularity in the calculus of variations},
Arch.~Ration.~Mech.~Anal. \textbf{95} (1986), 227--252.
\href{https://doi.org/10.1007/BF00251360}
{doi:10.1007/BF00251360}.

\bibitem{EvICM}
\textsc{L.C.~Evans},
\textit{Quasiconvexity and Partial Regularity in the Calculus of Variations},  
in Proceedings of the International Congress of Mathematicians
(Berkeley, California, 1986), Vol.~II,
Amer. Math. Soc., Providence, RI, 1987, 1064--1071.
\href{https://www.mathunion.org/fileadmin/ICM/Proceedings/ICM1986.2/ICM1986.2.ocr.pdf}
{ICM 1986 Proceedings}.



\bibitem{FO}
\textsc{K. F\"assler \& T. Orponen},
\textit{Dorronsoro's theorem in Heisenberg groups},
Bull. Lond. Math. Soc. \textbf{52} (2020), 472--488.
\href{https://doi.org/10.1112/blms.12341}
{doi:10.1112/blms.12341}.



\bibitem{GiaICM}
\textsc{M.~Giaquinta},
\textit{The problem of the regularity of minimizers},
in Proceedings of the International Congress of Mathematicians
(Berkeley, California, 1986), Vol.~II,
Amer.~Math.~Soc., Providence, RI, 1987, 1072--1083.
\href{https://www.mathunion.org/fileadmin/ICM/Proceedings/ICM1986.2/ICM1986.2.ocr.pdf}
{ICM 1986 Proceedings}.

\bibitem{GiaETH} 
\textsc{M.~Giaquinta}, 
\textit{Introduction to regularity theory for nonlinear elliptic systems.} 
Lectures in Mathematics, ETH Zürich. Basel: Birkhäuser Verlag. viii, 131 p. (1993).

\bibitem{GG82}
\textsc{M.~Giaquinta \& E.~Giusti},
\textit{On the regularity of the minima of variational integrals},
Acta Math. \textbf{148} (1982), 31--46, 
\href{https://doi.org/10.1007/BF02392725}
{doi:10.1007/BF02392725}.

\bibitem{Giusti}
\textsc{E.~Giusti},
\textit{Direct methods in the calculus of variations},
World Scientific, Singapore, 2003.

\bibitem{gm}
\textsc{F.~Gmeineder},
\textit{Partial regularity for symmetric quasiconvex functionals on BD},
J.~Math.~Pures et Appl. \textbf{145} (2021), 83--129,
\href{https://doi.org/10.1016/j.matpur.2020.09.005}
{doi:10.1016/j.matpur.2020.09.005}


\bibitem{gmjk1}
\textsc{F.~Gmeineder \& J.~Kristensen},
\textit{Partial regularity for BV minimizers},
Arch.~Ration.~Mech.~Anal. \textbf{232} (2019), 1429--1473.
\href{https://doi.org/10.1007/s00205-018-01346-5}
{doi:10.1007/s00205-018-01346-5}.

\bibitem{gmjk2}
\textsc{F.~Gmeineder \& J.~Kristensen},
\textit{Quasiconvex functionals of $(p,q)$-growth and the partial
regularity of relaxed minimizers},
Arch.~Ration.~Mech.~Anal. \textbf{248} (2024), Paper No.~80.
\href{https://doi.org/10.1007/s00205-024-02013-8}
{doi:10.1007/s00205-024-02013-8}.

\bibitem{agrt}
\textsc{A.~Guerra \& R.~Tione},
\textit{A singular minimizer of a smooth strongly polyconvex functional in two dimensions},
Preprint arXiv:2609.32404 [math.AP]  (2026).

  
\bibitem{Hamburger96}
\textsc{C.~Hamburger},
\textit{Partial regularity for minimizers of variational
integrals with discontinuous integrands},
Ann.~Inst.~H.~Poincar\'e Anal. Non Lin\'eaire
\textbf{13} (1996), 255--282.
\href{https://doi.org/10.1016/S0294-1449(16)30104-4}
{doi:10.1016/S0294-1449(16)30104-4}.

\bibitem{cirving}
\textsc{C.~Irving},
\textit{Partial regularity for minima of higher-order quasiconvex
integrands with natural Orlicz growth},
Ann.~Mat.~Pura Appl. \textbf{204} (2025), 1767--1793.
\href{https://doi.org/10.1007/s10231-025-01547-2}
{doi:10.1007/s10231-025-01547-2}.

\bibitem{ilr}
\textsc{C.~Irving, Z.~Li \& B.~Rai\textcommabelow{t}\u{a}},
\textit{Partial regularity for $\mathscr{A}$-quasiconvex variational problems of linear growth},
Preprint arXiv:2604.07538 [math.AP] (2026).



\bibitem{km05}
\textsc{J.~Kristensen \& G.~Mingione},
\textit{The singular set of $\omega$-minima},
Arch.~Ration.~Mech.~Anal. \textbf{177} (2005), 93--114.
\href{https://doi.org/10.1007/s00205-005-0361-x}
{doi:10.1007/s00205-005-0361-x}.

\bibitem{km06}
\textsc{J.~Kristensen \& G.~Mingione},
\textit{The singular set of minima of integral functionals},
Arch.~Ration.~Mech.~Anal. \textbf{180} (2006), 331--398.
\href{https://doi.org/10.1007/s00205-005-0402-5}
{doi:10.1007/s00205-005-0402-5}.

\bibitem{km07}
\textsc{J.~Kristensen \& G.~Mingione},
\textit{The singular set of Lipschitzian minima of multiple integrals},
Arch.~Ration.~Mech.~Anal. \textbf{184} (2007), 341--369.
\href{https://doi.org/10.1007/s00205-006-0036-2}
{doi:10.1007/s00205-006-0036-2}.

\bibitem{KMS}
\textsc{T.~Kuusi, G.~Mingione \& Y.~Sire},
\textit{Nonlocal self-improving properties},
Anal.~PDE \textbf{8} (2015), 57--114.
\href{https://doi.org/10.2140/apde.2015.8.57}
{doi:10.2140/apde.2015.8.57}.

\bibitem{min03}
\textsc{G.~Mingione},
\textit{The singular set of solutions to non-differentiable elliptic systems},
Arch.~Ration.~Mech.~Anal. \textbf{166} (2003), 287--301.
\href{https://doi.org/10.1007/s00205-002-0231-8}
{doi:10.1007/s00205-002-0231-8}.

\bibitem{min03-2}
\textsc{G.~Mingione},
\textit{Bounds for the singular set of solutions to non linear elliptic systems},
Calc.~Var.~Partial Differ.~Equ. \textbf{18} (2003), 373--400.
\href{https://doi.org/10.1007/s00526-003-0209-x}
{doi:10.1007/s00526-003-0209-x}.

\bibitem{min08}
\textsc{G.~Mingione},
\textit{Singularities of minima: a walk on the wild side of the
Calculus of Variations},
J.~Global Optim. \textbf{40} (2008), 209--223.
\href{https://doi.org/10.1007/s10898-007-9226-1}
{doi:10.1007/s10898-007-9226-1}.

\bibitem{Morrey}
\textsc{C.B.~Morrey Jr.},
\textit{Quasi-convexity and the lower semicontinuity of multiple integrals},
Pacific J.~Math. \textbf{2(1)} (1952), 25--53.
\href{https://doi.org/10.2140/pjm.1952.2.25}
{doi:10.2140/pjm.1952.2.25}.

\bibitem{sou}
\textsc{J. Sou\v{c}ek},
\textit{Singular solutions to linear elliptic systems},
Comment. Math. Univ. Carolin. \textbf{25} (1984), 273--281, 
\href{https://hdl.handle.net/10338.dmlcz/106299}
{DML-CZ: 106299}.

\bibitem{Widman}
\textsc{K.-O.~Widman},
\textit{H\"older continuity of solutions of elliptic systems},
Manuscripta Math. \textbf{5} (1971), 299--308.
\href{https://doi.org/10.1007/BF01367766}
{doi:10.1007/BF01367766}.


\end{thebibliography}
\end{document}